\documentclass[12pt, twoside, leqno]{article}
\usepackage{amsthm}
\usepackage{amssymb,latexsym}
\usepackage{enumerate}
\usepackage{amsfonts,amsmath,mathrsfs}
\usepackage{topcapt}
\usepackage[cp1250]{inputenc}
\usepackage[T1]{fontenc}
\usepackage{tabularx}
\usepackage{multirow}
\usepackage{epsfig}
\usepackage[figuresright]{rotating}
\usepackage{hyperref}
\hypersetup{hypertexnames=false}
\let\originalforall=\forall
\renewcommand{\forall}{\mathop{\vcenter{\hbox{\Large$\originalforall$}}}}

\let\originalexists=\exists
\renewcommand{\exists}{\mathop{\vcenter{\hbox{\Large$\originalexists$}}}}
\newtheorem{thm}{Theorem}[]
\newtheorem{cor}[thm]{Corollary}

\newtheorem{de}[thm]{Definiton}

\begin{document}

\baselineskip=17pt

\title{An elementary proof of the decomposition of measures on the circle group}

\author{Przemysław Ohrysko\\
Institute of Mathematics\\
Polish Academy of Sciences\\
00-956 Warszawa, Poland\\
E-mail: p.ohrysko@gmail.com}

\date{}

\maketitle

\renewcommand{\thefootnote}{}



\renewcommand{\thefootnote}{\arabic{footnote}}
\setcounter{footnote}{0}

\begin{abstract}
In this short note we give an elementary proof for the case of the
circle group of the theorem of O. Hatori and E. Sato which states that
every measure on the compact abelian group $G$ can be decomposed
into a sum of two measures with natural spectrum and a discrete
measure.
\end{abstract}
We consider the Banach algebra $M(\mathbb{T})$ of complex Borel
regular measures on the circle group with $M_{d}(\mathbb{T})$ as a
closed subalgebra consisting of all discrete (purely atomic
measures). Let $\mathfrak{M}(M(\mathbb{T}))$ denote the maximal
ideal space of $M(\mathbb{T})$, $\sigma(\mu)$ the spectrum of
$\mu\in M(\mathbb{T})$, $r(\mu)$ its spectral radius and
$\widehat{\mu}$ the Gelfand transform of $\mu$ (for basic facts in commutative harmonic analysis and Banach algebras see for example $\cite{katz}$ and $\cite{z}$) . We will treat the
Fourier - Stieltjes transform of a measure as a restriction of its
Gelfand transform to $\mathbb{Z}$ and so $\widehat{\mu}(n)$ will
stand for the $n$-th Fourier - Stieltjes coefficient
($n\in\mathbb{Z}$). It is an elementary fact that for all
$n\in\mathbb{Z}$ we have $\widehat{\mu}(n)\in\sigma(\mu)$ and
since the spectrum of an element in a unital Banach algebra is
closed we have an inclusion
$\overline{\widehat{\mu}(\mathbb{Z})}\subset\sigma(\mu)$. However, as was
first observed by Wiener and Pitt in $\cite{wp}$ this inclusion may be strict and
this interesting spectral behavior is now known as Wiener - Pitt
phenomenon (for the first precise proof of the existence of the Wiener - Pitt phenomenon see $\cite{schreider}$, for the most general result consult $\cite{wil}$ and $\cite{r}$, an alternative approach is presented in $\cite{graham}$). Therefore, after M. Zafran (see $\cite{Zafran}$), we introduce the following notion.
\begin{de}
We say that a measure $\mu\in M(\mathbb{T})$ has a natural
spectrum, if
\begin{equation*}
\sigma(\mu)=\overline{\widehat{\mu}(\mathbb{Z})}=\overline{\{\widehat{\mu}(n):n\in\mathbb{Z}\}}.
\end{equation*}
The set of all such measures will be denoted by
$\mathcal{N}(\mathbb{T})$.
\end{de}
The structure of $\mathcal{N}(\mathbb{T})$ is very complicated -
in particular, it was shown by M. Zafran for $I$-groups and by
O. Hatori and E. Sato in general case (see $\cite{h}$ and $\cite{hs}$), that this set is not closed under
addition (however, if one of summands has Fourier - Stieltjes coefficients tending to zero then the sum has a natural spectrum - check Theorem 20 in $\cite{ow}$). In order to obtain this result for the compact abelian
groups Hatori and Sato proved the following decomposition theorem.
\begin{thm}[Hatori,Sato]\label{hs}
Let $G$ be a compact abelian group. Then
\begin{equation*}
M(G)=\mathcal{N}(G)+\mathcal{N}(G)+M_{d}(G).
\end{equation*}
\end{thm}
The proof of this theorem in full generality relies on heavy
machinery of commutative harmonic analysis. However, as it will be
presented soon, the idea is in fact quite simple and this paper is
devoted to exhibit the core of the proof to wider audience.
Unfortunately, the group $G$ itself may be too convoluted and then
such an elementary arguments are not available. Hence we will
restrict our considerations to $G=\mathbb{T}$.
\\
The main tool for us will be the classical approximation theorem
of Kronecker (for a proof consult $\cite{katz}$).
\begin{thm}
Let $\alpha,\beta\in (0,2\pi)$ be real numbers and assume that
$\alpha,\beta,\pi$ are linearly independent over the field of
rational numbers. Then, for every real numbers $x,y$ and any
$\varepsilon>0$ there exists an integer $n$ such that
\begin{equation*}
|e^{in\alpha}-e^{ix}|<\varepsilon\text{ and
}|e^{in\beta}-e^{iy}|<\varepsilon.
\end{equation*}
\end{thm}
Now, we use the Kronecker theorem to analyze the set of values of
the Fourier - Stieltjes transforms of some discrete measures.
\begin{cor}\label{wni}
Let $\alpha,\beta$ be real numbers such that $\alpha,\beta,\pi$
are linearly independent over $\mathbb{Q}$. Then for
$\rho:=\frac{\delta_{\alpha}+\delta_{\beta}}{2}$ the following
holds
\begin{equation*}
\overline{\widehat{\rho}(\mathbb{Z})}=\overline{\mathbb{D}}=\{z\in\mathbb{C}:|z|\leq
1\}.
\end{equation*}
Moreover,
\begin{equation*}
\overline{\widehat{\rho}(2\mathbb{Z})}=\overline{\widehat{\rho}(2\mathbb{Z}+1)}=\overline{\mathbb{D}}.
\end{equation*}
\end{cor}
\begin{proof}
From the Kronecker theorem we have
\begin{equation*}
\overline{\{(e^{in\alpha},e^{in\beta}):n\in\mathbb{Z}\}}=\{z\in\mathbb{C}:|z|=1\}\times\{u\in\mathbb{C}:|u|=1\}.
\end{equation*}
Now it is a straightforward calculation that the mapping
\begin{equation*}
f:\{z\in\mathbb{C}:|z|=1\}\times\{u\in\mathbb{C}:|u|=1\}\mapsto\overline{\mathbb{D}}
\end{equation*}
defined by the formula
\begin{equation*}
f(z,u)=\frac{z+u}{2}\text{  is onto},
\end{equation*}
which gives the first assertion. For the second fact we observe
that
\begin{equation*}
\frac{(\widehat{\delta_{\alpha}}+\widehat{\delta_{\beta}})(2n)}{2}=\frac{e^{-2in\alpha}+e^{-2in\beta}}{2}=\frac{(\widehat{\delta_{2\alpha}}+\widehat{\delta_{2\beta}})(n)}{2}
\end{equation*}
and since numbers $2\alpha,2\beta,\pi$ are also independent over
$\mathbb{Q}$ the result follows once more from the Kronecker
theorem. For the last part we easily check that the function
$g(z,u)=\frac{ze^{-i\alpha}+ue^{-i\beta}}{2}$ defined on the same
domain as $f$ is onto $\overline{\mathbb{D}}$.
\end{proof}
We are ready now to give an elementary proof of Theorem $\ref{hs}$
for the case $G=\mathbb{T}$.
\\
Let $\alpha,\beta\in (0,2\pi)$ be such that the numbers
$\alpha,\beta,\pi$ are linearly independent over $\mathbb{Q}$ and
put $\rho=\frac{\delta_{\alpha}+\delta_{\beta}}{2}$,
$\theta_{0}=\frac{\delta_{0}+\delta_{\pi}}{2}$,
$\theta_{1}=\frac{\delta_{0}-\delta_{\pi}}{2}$. For any $\mu\in
M(\mathbb{T})$ we define $\mu_{0}=\mu\ast\theta_{0}$,
$\mu_{1}=\mu-\mu_{0}=\mu\ast\theta_{1}$ and $r_{0},r_{1}$ equal to
$r(\mu_{0}),r(\mu_{1})$ (respectively). Moreover, we put
\begin{gather*}
\nu_{0}=\mu_{0}+r_{0}\rho\ast\theta_{1},\\
\nu_{1}=\mu_{1}+r_{1}\rho\ast\theta_{0},\\
\nu_{2}=-r_{0}\rho\ast\theta_{1}-r_{1}\rho\ast\theta_{0}.
\end{gather*}
Then $\mu=\nu_{0}+\nu_{1}+\nu_{2}$ and $\nu_{2}\in
M_{d}(\mathbb{T})$. It is enough to show that
$\nu_{0}\in\mathcal{N}(\mathbb{T})$ (second part is obtained
analogously). Since
\begin{equation*}
\widehat{\rho\ast\theta_{1}}(\mathbb{Z})=\widehat{\rho}(2\mathbb{Z}+1)\cup\{0\}
\end{equation*}
we have from Corollary $\ref{wni}$
\begin{equation}\label{cal}
r_{0}\overline{\mathbb{D}}=r_{0}\overline{\widehat{(\rho\ast\theta_{1})}(\mathbb{Z})}.
\end{equation}
Remembering that $(\delta_{0}-\delta_{\pi})\ast
(\delta_{0}+\delta_{\pi})=0$ we get $\theta_{1}\ast\theta_{0}=0$
which leads to
\begin{equation*}
\mu_{0}\ast\rho\ast\theta_{1}=0
\end{equation*}
and so
\begin{equation*}
\widehat{\nu_{0}}(\mathbb{Z})\subset
\widehat{\mu_{0}}(\mathbb{Z})\cup
r_{0}\widehat{(\rho\ast\theta_{1})}(\mathbb{Z}).
\end{equation*}
By the definition of the spectral radius and $(\ref{cal})$ we have
\begin{equation*}
\overline{\widehat{\nu_{0}}(\mathbb{Z})}\subset
r_{0}\overline{\widehat{(\rho\ast\theta_{1})}(\mathbb{Z})}=r_{0}\overline{\mathbb{D}}.
\end{equation*}
On the other hand, examining the formula for $\nu_{0}$ we get
\begin{equation}\label{zaw}
r_{0}\widehat{\rho}(2\mathbb{Z}+1)=\widehat{\nu_{0}}(2\mathbb{Z}+1)\subset\widehat{\nu_{0}}(\mathbb{Z}).
\end{equation}
Combining $(\ref{cal})$ and $(\ref{zaw})$ we obtain
\begin{equation}\label{row}
\overline{\widehat{\nu_{0}}(\mathbb{Z})}=r_{0}\overline{\widehat{(\rho\ast\theta_{1})}(\mathbb{Z})}=r_{0}\overline{\mathbb{D}}.
\end{equation}
Recalling once again that $\theta_{0}\ast\theta_{1}=0$ we have for
every $\varphi\in\mathfrak{M}(M(\mathbb{T}))$ either
$\varphi(\theta_{0})=0$ or $\varphi(\theta_{1})=1$. If
$\varphi(\theta_{0})=0$ then
$\varphi(\nu_{0})=r_{0}\varphi(\rho)\varphi(\theta_{0})$ and
\begin{equation*}
|\varphi(\nu_{0})|\leq r_{0} r(\rho)r(\theta_{0})\leq r_{0}
||\rho||\cdot ||\theta_{0}||=r_{0}.
\end{equation*}
In the second case, $\varphi(\theta_{1})=0$ which provides
$|\varphi(\nu_{0})|\leq r_{0}$. Since the spectrum of a measure is
an image of its Gelfand transform $\sigma(\nu_{0})\subset
r_{0}\overline{\mathbb{D}}$. But from $(\ref{row})$
\begin{equation*}
r_{0}\overline{\mathbb{D}}=\overline{\widehat{\nu_{0}}(\mathbb{Z})}\subset
\sigma(\nu_{0})
\end{equation*}
which finally gives
$\sigma(\nu_{0})=\overline{\widehat{\nu_{0}}(\mathbb{Z})}$ and
finishes the proof.
\\
I would like to express my warm gratitude to the professor Michał Wojciechowski for invaluable suggestions and ideas.


\begin{thebibliography}{HD}
\normalsize
\baselineskip=17pt
\bibitem [G]{graham} C.C. Graham: \textit{A Riesz Product Proof of the Wiener-Pitt Theorem}, Proc. Amer. Math. Soc., vol. 44, no. 2, pp. 312-314, 1974
\bibitem [H]{h} O. Hatori: \textit{Measures with Natural Spectra on Locally Compact Abelian Groups}, Proc. Amer. Math. Soc., vol. 126, pp. 2351-2353, 1998.
\bibitem [HS]{hs} O. Hatori, E. Sato: \textit{Decompositions of Measures on Compact Abelian Groups}, Tokyo J. of Math., vol. 24, no. 1, pp. 13-18, 2001.
\bibitem [K]{katz} Y. Katznelson: \textit{An Introduction to Harmonic Analysis}, Cambridge Univeristy Press, 2004.
\bibitem [OW]{ow} P. Ohrysko, M. Wojciechowski: \textit{On the relathionships between Fourier - Stieltjes coefficients and spectra of measures} Studia Math, vol. 221, pp. 117-140, 2014.
\bibitem [R]{r} W. Rudin: \textit{Fourier Analysis on Groups}, Wiley Classics Library, 1990.
\bibitem [S]{schreider} Y. A. \v{S}reider: \textit{The structure of maximal ideals in rings of measures with convolution}, Mat. Sbornik, vol. 27, pp. 297-318, 1950; Amer. Math. Soc. Trans., no. 81, 1953.
\bibitem [W]{wil} J.H. Williamson: \textit{A theorem on algebras of measures on topological groups}, Proc. Edinburh Philos. Soc., vol. 11, pp. 195-206, 1959.
\bibitem [WP]{wp} N. Wiener, H.R. Pitt: \textit{Absolutely Convergent Fourier-Stieltjes Transforms}, Duke Math. J., vol. 4, no. 2, pp. 420-436, 1938.
\bibitem [Z]{Zafran} M. Zafran: \textit{On Spectra of Multipliers}, Pacific Journal of Mathematics, vol. 47, no. 2, 1973.
\bibitem [Ż]{z} W. Żelazko: \textit{Banach Algebras}, Elsevier Science Ltd and PWN, February 1973.
\end{thebibliography}
\end{document}